\documentclass[a4paper]{amsart}
\usepackage[utf8x]{inputenc}
\usepackage{amsmath, amsfonts, amsthm, amssymb, verbatim, enumitem, mathtools}
\usepackage[colorlinks=true,linkcolor=blue,citecolor=red]{hyperref}
\usepackage[ruled,vlined,linesnumbered]{algorithm2e}
\usepackage{color}
\usepackage{thmtools}
\usepackage{thm-restate}

\newcommand{\rank}{\operatorname{rank}}
\newcommand{\im}{\operatorname{im}}
\newcommand{\N}{\mathbb{N}}
\newcommand{\R}{\mathbb{R}}

\DeclareRobustCommand{\stirling}{\genfrac\{\}{0pt}{}}

\DeclarePairedDelimiter\abs{\lvert}{\rvert}%

\newtheorem{thm}{Theorem}[section]
\newtheorem{lem}[thm]{Lemma}
\newtheorem{prop}[thm]{Propostion}
\newtheorem{cor}[thm]{Corollary}
\theoremstyle{definition}
\newtheorem{definition}[thm]{Definition}

\let\oldmarginpar\marginpar
\renewcommand\marginpar[1]{\-\oldmarginpar[\raggedleft\footnotesize #1]%
{\raggedright\footnotesize #1}}

\title{Random ubiquitous transformation semigroups}
\author{Julius Jonu\v{s}as}
\address{Julius Jonu\v{s}as, Mathematical Institute, University of St Andrews, St Andrews, KY16 9SS, UK.}
\email{julius.jonusas@st-andrews.ac.uk}

\author{Sascha Troscheit}
\address{Sascha Troscheit, Department of Pure Mathematics, University of Waterloo, 200 University Avenue West, Waterloo, N2L 3G1, Canada.}
\email{stroscheit@uwaterloo.ca}

\thanks{We want to thank J.~D.~Mitchell for providing the experimental data} 
\thanks{ST was initially supported by EPSRC DTG EP/K503162/1}
\begin{document}
\maketitle

\begin{abstract}
  A \emph{smallest generating set} of a semigroup is a generating set of the
  smallest cardinality. Similarly, an \emph{irredundant generating set} $X$ is
  a generating set such that no proper subset of $X$ is also a generating set.
  A semigroup $S$ is \emph{ubiquitous} if every irredundant generating set of
  $S$ is of the same cardinality. 
  
  We are motivated by a na\"{i}ve algorithm to find a small generating set for
  a semigroup, which in practice often outputs a smallest generating set. We
  give a sufficient condition for a transformation semigroup to be ubiquitous
  and show that a transformation semigroup generated by $k$ randomly chosen
  transformations asymptoticly satisfies the sufficient condition. Finally, we
  show that under this condition the output of the previously mentioned
  na\"{i}ve algorithm is irredundant. 
\end{abstract}

\section{Introduction}

A generating set $X$ of a semigroup $S$ is a \emph{smallest generating set},
also known as \emph{minimum generating set}, if every subset of $S$ with
cardinality strictly smaller than $|X|$ does not generate $S$. The size of a
smallest generating set is known as the \emph{rank of $S$}. Similarly, an
\emph{irredundant generating set for $S$} is a generating set $X$ such that no
proper subset of $X$ is a generating set for $S$. Of course, the notions of a
smallest generating set, irredundant generating set, and the rank have a
natural interpretation for groups and other algebraic objects. The question of
finding a smallest generating set or a rank is a classical one, see for
example~\cite{Arvind:2006aa,Papadimitriou:1996aa} in the case of quasigroups
and~\cite{Gomes:1987aa,Gomes:1992aa,Gray:2014aa} in the case of semigroups.
However, from a computational perspective this is, in general, not an easy
problem. In particular, there is no known efficient algorithm to find the rank
of a given $S$, besides examining most of its subsets. As such, fast
na\"{i}ve algorithms are sometimes used to obtain small, but not necessarily
smallest, generating sets. The simplest of them is Algorithm~\ref{algorithm:1}.

\begin{algorithm}\label{algorithm:1}
  \SetKwInOut{Input}{Input}\SetKwInOut{Output}{Output}
  \Input{A list $S$ of all the elements of a semigroup}
  \Output{A generating set $X$}
  \BlankLine
  $X \longleftarrow \varnothing$\;
  \While{$|\langle X \rangle| \neq |S|$}{
    \For{$s \in S$}{
      \If{$s \notin \langle X \rangle $}{
        $X \longleftarrow X \cup \{ s \}$ \;
      }
    }
  }
  \caption{Greedy}
\end{algorithm}

The algorithm applies to both groups and semigroups. The advantages of the
Greedy algorithm are its speed and that it requires no \textit{a priori}
knowledge about the object. The latter might be seen as a drawback
if some structural information is known.
For semigroups this algorithm can be improved by taking into account its
$\mathcal{J}$-class structure.

In order to define the next algorithm, we require some notation. Let $S$ be a
semigroup and let $1$ be a symbol which is not in $S$. Define $S^1 = S \cup
\{1\}$ to be a semigroup such that for all $x, y \in S$ the product $x \cdot y$
in $S^1$ is the same as the product in $S$, and $x \cdot 1 = 1 \cdot x = x$ for
all $x \in S^1$. 
It is
routine to verify that the operation $[\cdot]$ on $S^1$ is associative. If $A, B
\subseteq S^1$ define $A x = \{ a \cdot x : a \in A\}$ and similarly define $x
A$ and $A x B$. Define a relation on $S$ by 
\[
  x \leq y \quad \text{if and only if} \quad S^1 x S^1 \subseteq S^1 y S^1.
\]
Then $\leq$ is reflexive and transitive, however it might fail to be
antisymmetric. In other words, $\leq$ is a \emph{preorder} on $S$. Clearly, the relation
\[
  a \mathcal{J} b \quad \text{ if and only if} \quad a \geq b \; \text{and} \; a \leq b
\]
is an equivalence relation on $S$ and the preorder $\leq$ induces a partial
order on the equivalence classes of $\mathcal{J}$, which we will also denote by
$\leq$ if the distinction is clear from the context. We say that the list
$s_{1, 1}, \ldots, s_{1, n_1}, s_{2, 1}, \ldots, s_{2, n_2} \ldots s_{k, n_k}$
of all elements of $S$ is ordered according to the preorder $\leq$ if $\{s_{i,
1}, \ldots, s_{i, n_i}\}$ is an equivalence class of $\mathcal{J}$ for all $i$
and if $s_{i, k} \leq s_{j, m}$ then $i \leq j$. Using this idea we can state
Algorithm~\ref{algorithm:2}.

\begin{algorithm}\label{algorithm:2}
  \SetKwInOut{Input}{Input}\SetKwInOut{Output}{Output}
  \Input{A semigroup $S$} 
  \Output{A generating set $X$}
  \BlankLine
  $L \longleftarrow$ order elements of $S$ according to the preorder $\leq$\; 
  $X \longleftarrow Greedy(L)$\; 
  \caption{SmallGeneratingSet, (implemented in \emph{Semigroups}~\cite{Mitchell2016aa} for GAP~\cite{:2017aa})}
\end{algorithm}

If $S$ is a group, then $S^1 x S^1 = S$ for every $x \in S$. Hence there is a
single $\mathcal{J}$-class in $S$ and so any permutation of elements of $S$ is
ordered according to the preorder~$\leq$, and so the SmallGeneratingSet
algorithm does not perform any better than Greedy. For proper semigroups the algorithm is
particularly useful if the $\mathcal{J}$-class structure is known in advance,
for example if the semigroup was enumerated using the \emph{Froidure-Pin
algorithm}~\cite{Froidure:1997fr,Jonusas:2017aa} or algorithms appearing
in~\cite{East:aa}. It is easy to come up with examples for which
SmallGeneratingSet might return a generating set which is not a smallest
generating set or even an irredundant generating set, for example any
non-trivial finite group $G$. Even though the algorithm is na\"{i}ve, it
performs surprisingly well in practice. For instance, we ran the
SmallGeneratingSet algorithm on all $836\,021$ semigroups (up to
(anti-)isomorphism) of size $7$, available in
\emph{SmallSemi}~\cite{Distler:2016aa}. In all cases the generating set found
was a smallest generating set.

Let $n \in \N$ and let $\mathcal{T}_n$ be the \emph{transformation monoid} on
$n$ points, that is the set of all functions from $\{1, \ldots, n\}$ to itself.
The set $\mathcal{T}_n$ forms a semigroup under the composition of functions. In
the following table, we consider every subgroup of $\mathcal{T}_3$ up to
conjugation. Observe that --- for most of them --- the size of the generating
set output by SmallGeneratingSet is equal to the rank or is one greater.
\begin{table}[h]
\begin{minipage}{.6\linewidth}
\begin{center}
\caption{Subsemigroups of $\mathcal{T}_3$}
\begin{tabular}{ c || c | c | c | c | c | c | c }
      & \multicolumn{7}{ c}{Size of the output} \\ 
    Rank  & \multicolumn{1}{c}{1}
          & \multicolumn{1}{c}{2}
          & \multicolumn{1}{c}{3}
          & \multicolumn{1}{c}{4}
          & \multicolumn{1}{c}{5}
          & \multicolumn{1}{c}{6}
          & \multicolumn{1}{c}{7}  \\ \hline\hline
    1 & 7 & 3  & 1  & 0  & 0  & 0  & 0 \\ \cline{2-8}
    2 & - & 32 & 25 & 11 & 3  & 1  & 0 \\ \cline{2-8} 
    3 & - & -  & 38 & 50 & 23 & 9  & 2 \\ \cline{2-8}
    4 & - & -  & -  & 23 & 28 & 6  & 6 \\ \cline{2-8} 
    5 & - & -  & -  & -  & 5  & 7  & 2 \\ 
\end{tabular}
\end{center}
\end{minipage}%
\end{table}

The main motivation for this paper is to provide mathematical justification as
to why SmallGeneratingSet algorithms often returns a smallest generating set.
In order to do so, we consider properties of transformation semigroups picked
at random, in a certain way. We say that a semigroup $S$ is \emph{ubiquitous}
if every irredundant generating of $S$ is also a smallest generating set.
Alternatively, if $r$ is the rank of $S$, then $S$ is ubiquitous if every
irredundant generating set is of size $r$.

First we will provide a sufficient condition for a transformation semigroup to
be ubiquitous.

\begin{restatable}{thm}{ubiquitous}\label{thm:ubiquitous}
  Let $S \leq \mathcal{T}_n$ and suppose that $X$ is a generating set for $S$
  such that $\rank(xyz) < \rank(y)$ for all $x, y, z \in X$. Then $S$ is
  ubiquitous.
\end{restatable}

Even though we restrict our attention to transformation semigroups in this paper,
Theorem~\ref{thm:ubiquitous} can be generalised to include semigroups of
partial bijections as well. 
We follow the approach of Cameron~\cite{Cameron:2013aa} of choosing a random
transformation semigroup. That is for some $k \geq 1$ we choose $k$
transformations of degree $n$  with uniform probability and consider the
semigroup generated by them. We show that most transformation semigroups are
ubiquitous.

\begin{thm}\label{thm:main}
  Let $k \geq 1$, and let $\mathbb{P}_k(n)$ be the probability that for $x_1,
  \ldots, x_k \in \mathcal{T}_n$, chosen with uniform probability, the
  semigroup $\langle x_1, \ldots, x_k \rangle$ is ubiquitous. Then
  $\mathbb{P}_k(n) \to 1$ as $n \to \infty$ exponentially fast.
\end{thm}

Even though SmallGeneratingSet does not return an irredundant generating set in
general, we show that under the assumptions of Theorem~\ref{thm:ubiquitous} the
output is irredundant. Hence the final result of the paper is as follows.

\begin{thm}\label{thm:smallgen}
  Let $k \geq 1$, and let $\mathbb{W}_k(n)$ be the probability that for $x_1,
  \ldots, x_k \in \mathcal{T}_n$, chosen with uniform probability,
  SmallGeneratingSet returns a smallest generating set for a semigroup $\langle
  x_1, \ldots, x_k \rangle$. Then $\mathbb{W}_k(n) \to 1$ as $n \to \infty$
  exponentially fast.
\end{thm}

Here we only look at the asymptotic behaviour of transformation semigroups,
however the same question can be investigated for any other infinite family of
semigroups, for example symmetric inverse monoids on $\{1,\ldots, n\}$, or
binary relations on $n$ points. 

\section{Preliminaries}\label{section:pre}

In this section we give the definitions and notation needed in the remainder of
the paper.

\begin{definition}\label{defn:GreensRels}
Let $S$ be a semigroup and let $x, y \in S$. The \emph{Green's relations}
$\mathcal{L}$, $\mathcal{R}$, $\mathcal{J}$, and $\mathcal{D}$ are the following
equivalence relations on $S$: 
\begin{align*}
  x \mathcal{L} y \quad &\text{if and only if} \quad S^1 x = S^1 y\\ 
  x \mathcal{R} y \quad &\text{if and only if} \quad x S^1 = y S^1\\
  x \mathcal{J} y \quad &\text{if and only if} \quad S^1 x S^1 = S^1 y S^1
\end{align*}
and $\mathcal{D}$ is the smallest equivalence relation containing both
$\mathcal{L}$ and $\mathcal{R}$.
\end{definition}

Let $x \in S$. Then $L_x$, $R_x$, and $D_x$ denote the equivalences classes of
$\mathcal{L}$, $\mathcal{R}$, and $\mathcal{D}$, respectively, containing $x$.
If $S$ is finite, then $\mathcal{D} = \mathcal{J}$, for a proof
see~\cite{Howie:1995fk}. Since we are only interested in finite semigroups
we will not make any distinction between the $\mathcal{D}$
and $\mathcal{J}$ relations.

Throughout the paper, we write elements of $\mathcal{T}_n$ on the right of
their argument and we write functions from a subset of $\mathbb{R}^n$ to
$\mathbb{R}$ on the left. This is done in agreement with two different
notations prevalent in algebra and analysis.

Let $f \in \mathcal{T}_n$, and let $A \subseteq \{1, \ldots, n\}$. Then 
\[
  (A)f = \{(a)f : a \in A\}
\]
and the \emph{image of $f$} is the set $\im(f) = (\{1, \ldots, n\})f$. A
\emph{transversal of $f$} is a set $\mathfrak{T} \subseteq \{1, \ldots, n\}$
such that $f$ is injective on $\mathfrak{T}$ and $(\mathfrak{T})f = \im(f)$.
The \emph{rank of $f$} is $\rank(f) = |\im(f)| = |\mathfrak{T}|$, where
$\mathfrak{T}$ is a transversal of $f$. The \emph{kernel of $f$}, denoted by
$\ker(f)$, is the equivalence relation defined by
\[
  (x, y) \in \ker(f) \quad \text{if and only if} \quad (x)f = (y)f.
\]
Hence a \emph{kernel class of $f$} containing $x \in \{1, \ldots, n\}$ is the set 
\[
  \{y \in \{1, \ldots, n\} : (y)f = (x)f \}. 
\] 

Using the above definition we can state a classical result describing Green's
classes of transformation semigroups. The proof is easy and thus omitted. 

\begin{lem}\label{lem:greens-transformation}
  Let $S \leq \mathcal{T}_n$, and let $f, g \in S$. Then
  \begin{enumerate}[label = \rm{(\roman*)}]
    \item if $f \mathcal{L} g$ then $\im(f) = \im(g)$; 

    \item if $f \mathcal{R} g$ then $\ker(f) = \ker(g)$; 

    \item if $f \mathcal{D} g$ then $\rank(f) = \rank(g)$.
  \end{enumerate}
\end{lem}

For $n, r\in \mathbb{N}$ such that $r \leq n$, define $\mathcal{A}(n,r)$ to be
the set of partitions of $\{1,\ldots,n\}$ into $r$ non-empty components. 

\begin{lem}\label{lem:idempotent-numbers}
  Let $n, r \in \mathbb{N}$ such that $r \leq n$. Then
  \[
    \sum_{\{A_1, \ldots, A_r\} \in \mathcal{A}(n,r)}\;\; \prod_{i = 1}^r |A_i| =
    \binom{n}{r} r^{n - r}.
  \]
\end{lem}
\begin{proof}
  A function $f \in \mathcal{T}_n$ is called idempotent if $f^2 = f$. We prove
  the lemma by finding the number of idempotent transformation of
  $\mathcal{T}_n$ of rank $r$ in two ways. Denote this number by $N$. It can be
  shown that $f$ is an idempotent if and only if $(x)f = x$ for all $x \in
  \im(f)$.

  If $f \in \mathcal{T}_n$ is an idempotent of rank $r$, then there are
  $\binom{n}{r}$ choices for the $\im(f)$ and for every point in
  $\{1,\ldots,n\} \setminus \im(f)$ there are $r$ choices in $\im(f)$ to map to. Hence
  \[
    N = \binom{n}{r} r^{n - r}.
  \]

  On the other hand, the sets $A_1, \ldots, A_r$ are the kernel classes of $f
  \in \mathcal{T}_n$ if and only if $\{A_1, \ldots, A_r\} \in
  \mathcal{A}(n,r)$. If $f$ is an idempotent and $A_1, \ldots, A_r$ are kernel
  classes of $f$ then $(A_i)f \in A_i$ for all $i \in \{1, \ldots, r\}$, and so
  there are $\prod_{i = 1}^r |A_i|$ choices for the $\im(f)$. Hence
  \[
    N = \sum_{\{A_1, \ldots, A_r\} \in \mathcal{A}(n,r)} \prod_{i = 1}^r |A_i|,
  \]
  as required.
\end{proof}

Since $|\mathcal{A}(n, r)| = \stirling{n}{r}$, the following easy upper bound
for the Stirling numbers is an immediate consequence of
Lemma~\ref{lem:idempotent-numbers}.
\begin{cor}\label{lem:stirling-upper}
  Let $n, r \in \mathbb{N}$ be such that $r \leq n$. Then
  \[
    \stirling{n}{r} \leq \binom{n}{r} r^{n - r}.
  \]
\end{cor}

We will make use of Stirling's approximation formula
\[
  \sqrt{2\pi}n^{n + \frac{1}{2}}e^{-n} \leq n! \leq \sqrt{2\pi}n^{n +
  \frac{1}{2}}e^{-n + \frac{1}{2n}}.
\]
If $F : \mathbb{R} \to \mathbb{R}$, then we say that $G \in O(F)$
if there are $c > 0$ and $x_0 \in \mathbb{R}$ such that $\lvert G(x) \rvert \leq c
\lvert F(x) \rvert$ for all $x \geq x_0$. Then Stirling's formula can be
written as follows
\[
  \log n! =n\log n - n + O(\log(n)).
\]

Let $\mathbb{R}^+ = \{x \in \mathbb{R} : x > 0\}$. The final notion
required in this paper is the function $W : \mathbb{R}^+ \to \mathbb{R}^+$
defined so that 
\[
  x = W(x) e^{W(x)}
\]
for all $x \in \mathbb{R}^+$. Since the function $x \mapsto x e^x$ is strictly
increasing on $\mathbb{R}^+$, it follows that $W(x)$ is a well-defined function
on $\mathbb{R}^+$. In the literature $W(x)$ is known as \emph{Lambert W
function} or \emph{product logarithm}, see e.g.\ \cite{Corless96}. The value $\Omega = W(1)$ is known as
the \emph{omega constant} and it satisfies $\Omega e^\Omega = 1$, with the
numerical value $\Omega = 0.5671439\dots$\,.

\section{Sufficient condition for ubiquitous semigroups}

In this section we prove Theorem~\ref{thm:ubiquitous}. We will do so in a
series of lemmas. The first of which is the following easy observation about
products in the $\mathcal{D}$-classes. Recall that if $x, y \in S$, then by
$D_x$ we denote the $\mathcal{D}$-class containing $x$ and $x \leq y$ if and
only if $S^1 x S^1 \subseteq S^1 y S^1$.

\begin{lem}\label{lem:product-above}
  Let $S$ be a semigroup, and let  $z_1 \cdots z_m \in D_x$ where $x, z_1,
  \ldots, z_m \in S$. Then $x \leq z_i \cdots z_j$ under the preorder on $S$
  for all $i,j \in \{1, \ldots, m\}$ with $i \leq j$.
\end{lem}
\begin{proof}
  Let $x, z_1, \ldots, z_m \in S$ be such that $z_1 \cdots z_m \in D_x$. Then 
  \[
    S^1 x S^1 = S^1 z_1 \cdots z_m S^1 \subseteq S^1z_i \cdots z_jS^1,
  \]
  and so $x \leq z_i \cdots z_j$ by definition for all $i,j \in \{1, \ldots,
  m\}$ with $i \leq j$.
\end{proof}

Next we give a condition for a semigroup $S$ which restricts allowed products
in a given $\mathcal{D}$-class.

\begin{lem}\label{lem:different-ranks}
  Let $S \leq \mathcal{T}_n$, let $X$ be a generating set for $S$, and let $x
  \in X$ be such that $\rank(y_1xy_2) < \rank(x)$ for all $y_1, y_2 \in X$
  where $y_1, y_2 \geq x$. Then only the following products
  \[
    x, \quad y_1 \cdots y_m, \quad xy_1 \cdots y_m, \quad \text{or} \quad y_1
    \cdots y_m x
  \]
  where $m \geq 1$, $y_1, \ldots, y_m \in X\setminus \{x\}$ and $y_1, \ldots,
  y_m \geq x$ can be in $D_x$.
\end{lem}

\begin{proof}
  First observe that if $x^2 \in D_x$, then both $x$ and $x^2$ have the same
  rank by Lemma~\ref{lem:greens-transformation}, in other words $\abs{\im(x)} =
  \abs{\im(x^2)}$. However, since $x$ is a finite degree transformation and
  $\im(x^2) \subseteq \im(x)$, it follows that $\im(x) = \im(x^2)$, and so $x$
  acts as a bijection on $\im(x)$. Hence $\rank(x^3) = \rank(x)$, contradicting
  the hypothesis of the lemma. Therefore $x^2 \notin D_x$, and since $S^1 x^2
  S^1 \subseteq S^1 x S^1$, it follows that $x^2 < x$ under the preorder on
  $S$. Similarly, for every $y_1, y_2 \in X$ such that $y_1, y_2 \geq x$, it follows from
  Lemma~\ref{lem:greens-transformation} that $y_1 x y_2 < x$, since $\rank(y_1
  x y_2) < \rank(x)$ and $S^1 y_1 x y_2 S^1 \subseteq S^1 x S^1$.
  
  Let $z_1, \ldots, z_m \in X$ be such that $z_1 \cdots z_m \in D_x$. Then  $x
  \leq z_i$ for all $i$ by Lemma~\ref{lem:product-above}. Hence there are $k
  \in \N$, $n_1, \ldots, n_{k}, m_2, \ldots, m_{k} \geq 1$, and $m_1, m_{k + 1}
  \geq 0$ such that
  \[
    z_1 \cdots z_m = y_{1,1} \cdots y_{1, m_1} x^{n_1} y_{2, 1} \cdots  y_{k,m_k} x^{n_k}
    y_{k + 1, 1} \cdots y_{k + 1, m_{k + 1}}
  \]
  where $y_{i,j} \in X \setminus \{x\}$ and $y_{i, j} \geq x$ for all $i$ and
  $j$. Here we are assuming that
  \[
    m = \sum_{i = 1}^k n_i + \sum_{i = 1}^{k + 1} m_i,
  \]
  $z_1 = y_{1, 1}$, $z_2 = y_{1, 2}$, and so on. Again by
  Lemma~\ref{lem:product-above}, if $z = z_i \cdots z_j$ is a subproduct of
  $z_1 \cdots z_m$, then $x \leq z$. However $x^2 < x$, and thus $x^2$ is not a
  subproduct of $z_1 \cdots z_m$. That is, $n_{i} = 1$ for all $i \in \{1,
  \ldots, k\}$. Hence
  \[
    z_1 \cdots z_m = y_{1,1} \cdots y_{1, m_1} x y_{2, 1} \cdots
    x y_{k + 1, 1} \cdots y_{k + 1, m_{k + 1}}.
  \]
  In a similar fashion, if $y_{i, m_i}, y_{i + 1, 1} \in X\setminus\{x\}$ for
  $y_{i, m_i}, y_{i + 1, 1}\geq x$, then as observed above $y_{i, m_i} x y_{i +
  1, 1} < x$, and so $y_{i, m_i}xy_{i + 1, 1}$ is not a subproduct of $z_1
  \cdots z_m$. Hence $z_1 \cdots z_m$ is one of the following products  
  \[
    x, \quad y_1 \cdots y_m, \quad xy_1 \cdots y_m, \quad y_1
    \cdots y_m x, \quad \text{or} \quad x y_1
    \cdots y_m x
  \]
  where $m \geq 1$, $y_1, \ldots, y_m \in X\setminus \{x\}$ and $y_1, \ldots,
  y_m \geq x$. Hence it remains to show that $x y_1 \cdots y_m x \notin D_x$. 

  Suppose that $x y_1 \cdots y_m x \in D_x$ for some $m \geq 1$,  $y_1, \ldots,
  y_m \in X \setminus \{x\}$ such that $y_1, \ldots, y_m \geq x$. Then there
  are $a, b \in S^1$ such that $a x y_1 \cdots y_m x b = x$. Note that unless
  $a = b = 1$, the product $a x y_1 \cdots y_m x b$ is not in one of the above
  forms, and so cannot be in $D_x$. Hence $a = b = 1$, and thus $x y_1 \cdots
  y_m x = x$. Which is only possible if $y_1 \ldots y_m$ acts bijectively on
  $\im(x)$. Thus $y_1$ acts bijectively on $\im(x)$. If $\im(y_m x) = \im(x)$,
  then it follows that 
  \[
    \rank(y_m x y_1) = \abs{\im(y_m x y_1)} = \abs{\im(y_m x)} = \abs{\im(x)} =
    \rank(x),
  \]
  which contradicts the hypothesis of the lemma. Hence $\im(y_m x) \subsetneq
  \im(x)$. However, it then follows that 
  \[
    \rank(x y_1 \cdots y_m x) \leq |\im(y_m x)| < \rank(x),
  \]
  contradicting $x y_1 \cdots y_m x = x$. Therefore $x y_1 \cdots y_m x \notin
  D_x$ for all $m \geq 1$ and all $y_1, \ldots, y_m \in X \setminus \{x\}$ such
  that $y_1, \ldots y_m \geq x$, as required. 
\end{proof}

We prove a corollary in the case where the generating set is irredundant. 

\begin{cor}\label{cor:different-ranks2}
  Let $S \leq \mathcal{T}_n$, let $X$ be an irredundant generating set for $S$,
  and let $x \in X$ be such that $\rank(z_1xz_2) < \rank(x)$ for all $z_1, z_2
  \in X$ where $z_1, z_2 \geq x$. Then $pxuys \notin D_x$ for all $p, u, s \in
  S^1$ and any $x, y \in X$ such that $x \mathcal{D} y$.
\end{cor}

\begin{proof}
  Let $Y = \{z \in X\setminus \{x\} : z \geq x \}$. It follows from
  Lemma~\ref{lem:different-ranks} that every product of elements of $X$ which
  is in $D_x$ is of the form $x$, $z$, $xz$, or $zx$ for some product $z$ of
  elements of $Y$. In particular, $x \notin Y$ as $X$ is irredundant, and so
  $x$ occurs at most once in the product. If $x, y \in X$, $x\mathcal{D}y$, and
  $pxuys \in D_x = D_y$ for some $p, u, s \in S^1$, then there are $a, b, c, d
  \in S^1$ such that  
  \[
    axsyb = x \quad \text{and} \quad cxsyd = y.
  \]
  Hence $axscxsydb = x \in D_x$, but $x$ occurs twice in the product, which is
  a contradiction. 
\end{proof}

Finally, we prove Theorem~\ref{thm:ubiquitous}. Observe that if a
transformation semigroup $S \leq \mathcal{T}_n$ is such that all irredundant
generating sets have the same cardinality, then every irredundant generating
set is a smallest generating set. 

\ubiquitous*

\begin{proof}
  Let $X' \subseteq X$ be irredundant. Then $\rank(xyz) < \rank(y)$ for all $x,
  y, z \in X'$. It is sufficient to show that every irredundant generating set
  is of the same cardinality. Moreover, without loss of generality we may
  assume that $X$ is irredundant and show that every irredundant generating set
  is of size $\abs{X}$.

  Let $Y$ be an irredundant generating set for $S$. Let $\leq_d$ be a total
  order defined on $\mathcal{D}$-classes of $S$ such that if $D$ and $D'$ are
  $\mathcal{D}$-classes of $S$ and $D \leq D'$ under the partial order of
  $\mathcal{D}$-classes, then $D \leq_d D'$. Let $\{D_1, \ldots, D_d\}$ be the
  set of all $\mathcal{D}$-classes of $S$, indexed so that $D_d <_d \ldots <_d
  D_1$. For $k \in \{1, \ldots, d\}$, define 
  \[
    X_k = X \cap \left( \bigcup_{i = 1}^k D_i\right) \quad \text{and} \quad Y_k
    = Y \cap \left( \bigcup_{i = 1}^k D_i\right).
  \]
  Let $k \geq 1$ and let $z \in D_{k}$. By Lemma~\ref{lem:product-above} if
  $x_1 \cdots x_m \in D_k$ where $x_1, \ldots, x_m \in X$, then $z \leq x_i$,
  and so there is $j \leq k$ so that $x_i \in D_j$ for all $i \in \{1, \ldots,
  m\}$.
   In other words,
  \begin{equation}\label{equation:1}
    x_1 \cdots x_m \in D_k \; \text{where} \; x_1, \ldots, x_m \in X \implies x_i \in X_k
    \; \text{for all} \; i \in \{1, \ldots, m\}.
  \end{equation}
  The same argument applies to $Y$, and so 
  \begin{equation}\label{equation:2}
    D_k \subseteq \langle X_k \rangle \quad \text{and} \quad D_k \subseteq
    \langle Y_k \rangle
  \end{equation}
  for all $k \geq 1$.

  By the definition of the total order $\leq_d$, the $\mathcal{D}$-class
  $D_{1}$ is maximal, and so both $X$ and $Y$ intersect $D_1$ non-trivially.
  For any $i \geq 2$
  and $x_1, \ldots, x_i \in X_1$, it follows from
  Corollary~\ref{cor:different-ranks2} that $ x_{1} \cdots x_{i} \notin
  D_{x_1} = D_{1}$, and so $D_{1} = X_1$. Hence $Y_1 \subseteq X_1$ and
  $D_1 \cap \langle Y_1 \rangle = Y_1$. However, since $D_1$ is a maximal
  $\mathcal{D}$-class, $D_1 \subseteq \langle Y_1 \rangle$, and thus $X_1 = D_1
  \subseteq Y_1$. In other words, $X_1 = Y_1 = D_1$.
  
  For $k \geq 1$, suppose that $\abs{X_k} = \abs{Y_k}$ and $\langle X_k \rangle
  = \langle Y_k \rangle$. If $X \cap D_{k + 1} = \varnothing$, then $D_{k + 1}
  \subseteq \langle X_k \rangle = \langle Y_k \rangle$. Hence $X_{k + 1} = X_k$
  and $Y_{k + 1} = Y_k$, and thus $\abs{X_{k + 1}} = \abs{Y_{k + 1}}$ and
  $\langle X_{k + 1} \rangle = \langle Y_{k + 1} \rangle$.

  Suppose that $X \cap D_{k + 1} \neq \varnothing$. Then $D_{k + 1}
  \not\subseteq \langle X_k \rangle = \langle Y_k \rangle$, and so $Y \cap D_{k
  + 1} \neq \varnothing$. Suppose that $t \geq 0$ is largest integer such that
  there is $X' \subseteq X\cap D_{k +1}$ and $Y' \subseteq Y \cap D_{k + 1}$
  with $\abs{X'} = \abs{Y'} = t$ and $\langle X_k, X' \rangle = \langle
  Y_k, Y' \rangle$. If $t = \abs{Y \cap D_{k + 1}}$ and $x \in X \cap D_{k + 1}
  \setminus X'$, then 
  \[
    x \in D_{k + 1} \subseteq \langle Y_{k + 1} \rangle =  \langle Y_k, Y'
    \rangle = \langle X_k, X' \rangle,
  \]
  by \eqref{equation:2}. However, this is impossible, since $X$ is irredundant
  and $x \notin X_k \cup X' \subseteq X$. Hence if $t = \abs{Y \cap D_{k + 1}}$
  then $X' = X \cap D_{k + 1}$, or in other words $X_{k + 1} = X_k \cup X'$ and
  $Y_{k + 1} = Y_k \cup Y'$. Therefore, $\abs{X_{k + 1}} = \abs{X_k} + t =
  \abs{Y_k} + t = \abs{Y_{k + 1}}$ and $\langle X_{k + 1} \rangle = \langle
  Y_{k + 1} \rangle$. We will now show that $t = \abs{Y \cap D_{k + 1}}$.

  Suppose that $t < \abs{Y \cap D_{k + 1}}$. Then there is $y \in Y \cap D_{k +
  1} \setminus Y'$ and  $y$ is equal to a product of elements of $X_{k + 1}$ by
  \eqref{equation:2}. It follows from Corollary~\ref{cor:different-ranks2} that
  if $x_1 \cdots x_m \in D_{k + 1}$ where $x_1, \ldots, x_m \in X$ then there
  is at most one $i \in \{1, \ldots, m\}$ such that  $x_i \in X \cap D_{k +
  1}$, otherwise some subword of $x_1 \cdots x_m$ would not be an element of
  $D_{k + 1}$.   Since $y \notin Y_k \cup Y'$, the irredundancy of $Y$ implies
  that $y \notin \langle Y_k, Y' \rangle = \langle X_k, X' \rangle$. It follows
  that $y = p_1 \cdots p_m x s_1 \cdots s_l$ for some $x \in X \cap D_{k + 1}
  \setminus X'$, $m, l \geq 0$ and $s_i, p_i \in X_k$. Hence $y \in \langle x,
  X_k \rangle$. Since $x, y \in D_{k + 1}$, it follows that there are $a, b \in
  S^1$ such that
  \[
    a p_1 \cdots p_m x s_1 \cdots s_l b = ayb = x.
  \]
  It follows from \eqref{equation:1} and the discussion above that $a, b \in
  \langle X_k \rangle^1$, and so $x \in \langle y, X_k \rangle$. 
  Moreover 
  \[
    x \in \langle y, X_k, X' \rangle \quad \text{and} \quad y \in \langle x,
    X_k, X' \rangle.
  \]
  Therefore $\langle x, X_k, X' \rangle = \langle y, X_k, X' \rangle = \langle
  y, Y_k, Y' \rangle$, since $\langle X_k, X' \rangle = \langle Y_k, Y'
  \rangle$. However $\abs{X' \cup \{x\}} = \abs{Y' \cup \{y\}} = t + 1$, which
  contradicts the maximality of $t$. Therefore $t = \abs{Y \cap D_{k + 1}}$ and
  by the previous paragraph $\langle X_{k + 1} \rangle = \langle Y_{k + 1}
  \rangle$ and $\abs{X_{k + 1}} = \abs{Y_{k + 1}}$.

  By induction it follows that $\langle X_k \rangle = \langle Y_k \rangle$ and
  $\abs{X_k} = \abs{Y_k}$ for all $k \in \{1, \ldots, d\}$. In particular, $X_d
  = X$ and $Y_d = Y$, and thus $\abs{X} = \abs{Y}$, as required.
\end{proof}

\section{SmallGeneratingSet}

In this section we return to the motivating question about the algorithm
SmallGeneratingSet. First, we note that SmallGeneratingSet might return a
generating set which is not irredundant. For example, if the semigroup under
investigation is a group of size at least $2$, the algorithm can first pick an
identity and so return a generating set which includes an identity. However, we
show that under the assumptions of Theorem~\ref{thm:ubiquitous} the generating
set returned by SmallGeneratingSet is irredundant.

\begin{lem}
  Let $S \leq \mathcal{T}_n$ and suppose that $X$ is a generating set for $S$
  such that $\rank(xyz) < \rank(y)$ for all $x, y, z \in X$. Then
  SmallGeneratingSet returns an  irredundant generating set.
\end{lem}

\begin{proof}
  Let $X = \{x_1, \ldots, x_m\}$ be the output of the algorithm, and assume
  that the elements were selected in the order they are listed. Suppose that $I
  \subseteq X$ is irredundant and let $x_i \in X \setminus I$. Since $x_i$ was
  selected by the algorithm, it means that 
  \[
    x_i \notin \langle x_1, \ldots, x_{i - 1} \rangle,
  \]
  and so there exists $x_j \in I$ such that  $x_i \mathcal{D} x_j$ and $j > i$, otherwise 
  $x_i \notin \langle I \rangle$.
  Without loss of generality, we can assume that $i$ is the largest integer
  such that $x_i \in X \setminus I$ and $x_i \mathcal{D} x_j$. Then there are
  $a_1, \ldots, a_{k_a}, b_1,\ldots, b_{k_b} \in I$ such that
  \[
    a_1 \cdots a_{k_a} x_i b_1 \cdots b_{k_b} = x_j.
  \]
  Since $x_j \notin \langle x_1, \ldots, x_i \rangle$, it follows that at least
  one of the $a_1, \ldots, a_{k_a}, b_1, \ldots, b_{k_b}$ is $x_k$ for some $k
  > i$. It follows from Lemma~\ref{lem:product-above} that $x_k \geq x_j$, and
  since $k > i$ implies that $x_k \nleq x_i$, we have that $x_k \mathcal{D}
  x_i$.

  Finally, there are $c_1, \ldots, c_{k_c}, d_1, \ldots, d_{k_d} \in I$ such
  that
  \[
    c_1 \cdots c_{k_c} x_j d_1 \cdots d_{k_d} = x_i,
  \]
  and so
  \[
    a_1 \cdots a_{k_a} c_1 \cdots c_{k_c} x_j d_1 \cdots d_{k_d}  b_1 \cdots
    b_{k_b} = x_j.
  \]
  Which contradicts Corollary~\ref{cor:different-ranks2} as at least one of
  $a_1, \ldots, a_{k_a}, b_1, \ldots, b_{k_b}$ is $x_k$. Therefore, $X = I$.
\end{proof}

The following result is then immediate from Theorem~\ref{thm:ubiquitous}.

\begin{cor}\label{cor:smallesgen}
  Let $S \leq \mathcal{T}_n$ and suppose that $X$ is a generating set for $S$
  such that $\rank(xyz) < \rank(y)$ for all $x, y, z \in X$. Then
  SmallGeneratingSet returns a smallest generating set.
\end{cor}

\section{Asymptotics}

The main aim of this section is to show that if for some fixed $k \geq 1$ we
choose $x_1, \ldots, x_k \in \mathcal{T}_n$ with uniform probability, then the
probability $\mathbb{P}_k(n)$ that $\langle x_1, \ldots, x_k \rangle$ is
ubiquitous and the probability $\mathbb{W}_k(n)$ that SmallGeneratingSet
returns a smallest generating set for $\langle x_1, \ldots, x_k \rangle$ both
tend to $1$ as $n$ increases.

\begin{lem}\label{lem:case-split}
  Let $X \subseteq \mathcal{T}_n$ be such that $\rank(xyz) = \rank(y)$ for some
  $x, y, z \in X$. Then one of the following holds:
  \begin{enumerate}[label = \rm(\roman*)]
    \item there is $x \in X$ such that $\langle x \rangle$ is a group;

    \item there are distinct $x, y \in X$ such that $\rank(xyx) = \rank(y)$;

    \item there are mutually distinct $x, y, z \in X$  such that $\rank(xyz) = \rank(y)$.
  \end{enumerate}
\end{lem}

\begin{proof}
  Suppose that $\rank(xyz) = \rank(y)$ for some $x, y, z \in X$ and suppose
  that not all $x$, $y$, and $z$ are distinct. If $x = y = z$, then $\rank(x^3)
  = \rank(x)$, which is only possible if $x$ acts bijectively on $\im(x)$.
  However, in that case $\langle x \rangle$ is a group. Hence we only need to
  consider the case that where exactly two of $x$, $y$, and $z$ are equal.

  Suppose that $x = y$. Then $\rank(y^2z) = \rank(y)$, and since
  \[
    \rank(y) \leq \rank(y^2) \leq \rank(y^2z) = \rank(y),
  \]
  it follows that $\rank(y^2) = \rank(y)$. Hence by an argument similar to
  above $\langle y \rangle$ is a group. The case $y = z$ can be dealt with in
  an almost identical fashion. Therefore, there are distinct $x, y \in X$ such
  that $\rank(xyx) = \rank(y)$.
\end{proof}

In order to show that $\mathbb{P}_k(n) \to 1$ as $n \to \infty$, for every $n
\in \mathbb{N}$, we define three probabilities:
\begin{align*}
  \mathbb{G}_n \;& \text{is the probability that} \, \langle x \rangle \,
  \text{is a group} \, \text{where} \, x \in \mathcal{T}_n\, \text{is chosen} \\
  &\text{randomly with uniform probability} \\ 
  \mathbb{T}_n \;& \text{is the probability that} \, \rank(xyx) = \rank(y)  \,
  \text{where} \, x, y \in \mathcal{T}_n \, \text{are} \\
  &\text{chosen randomly with uniform probability} \\ 
  \mathbb{V}_n \;& \text{is the probability that} \, \rank(xyz) = \rank(z)  \,
  \text{where} \, x, y, z \in \mathcal{T}_n \\
  &\text{are chosen randomly with uniform probability}. \\ 
\end{align*}
For a fixed $k \geq 1$, if $x_1, \ldots, x_k \in \mathcal{T}_n$ are chosen
randomly with uniform probability, it follows from Lemma~\ref{lem:case-split}
that the probability that there are $x, y, z \in \{x_1, \ldots, x_k\}$ such
that $\rank(xyz) = \rank(y)$ is bounded from above
by
\[
  k \mathbb{G}_n + k(k - 1) \mathbb{T}_n + k (k - 1) (k - 2) \mathbb{V}_n.
\]
Hence by Theorem~\ref{thm:ubiquitous}
\[
  \mathbb{P}_k(n) \geq 1 - k \mathbb{G}_n  - k(k-1) \mathbb{T}_n - k (k - 1) (k
  - 2) \mathbb{V}_n,
\]
and the same lower bound hold for $\mathbb{W}_k(n)$ by
Corollary~\ref{cor:smallesgen}. Hence in order to prove Theorems~\ref{thm:main}
and~\ref{thm:smallgen} it suffices to show that $\mathbb{G}_n \to 0$,
$\mathbb{T}_n \to 0$, and $\mathbb{V}_n \to 0$ as $n \to \infty$. We will do so
in the remaining three subsections of the paper.

\subsection{$\mathbb{G}_n$ tends to zero}

We begin by obtain an expression for $\mathbb{G}_n$ in terms of $n$.

\begin{lem}\label{lem:1-lower-bound}
  Let $n \in \mathbb{N}$. Then 
  \[
    \mathbb{G}_n = \frac{n!}{n^n}\sum_{k = 0}^{n - 1} \frac{(n - k)^k}{k!}. 
  \] 
\end{lem}
\begin{proof}
  First observe that for any $x \in \mathcal{T}_n$, the semigroup $\langle x
  \rangle$ is a group if and only if $x$ acts as a bijection on $\im(x)$. There
  are
  \[
    \sum_{r = 1}^n \binom{n}{r} r^{n - r} r!
  \] 
  transformations $x$ such that $x$ acts bijectively on $\im(x)$. That is, if
  $|\im(x)| = r$, then there are $\binom{n}{r}$ choices for $\im(x)$, $r!$ ways
  of bijectively mapping $\im(x)$ to itself, and $r^{n - r}$ ways to map every point from
  $\{1,\ldots,n\} \setminus \im(x)$ to $\im(x)$. Since $|\mathcal{T}_n| = n^n$, the
  probability of randomly choosing $x \in \mathcal{T}_n$ such that $\langle x
  \rangle$ is a group is
  \[
    \mathbb{G}_n =  \frac{1}{n^n} \sum_{r = 1}^n \binom{n}{r} r^{n - r} r! = 
    \frac{n!}{n^n}\sum_{r = 1}^n \frac{r^{n - r}}{(n - r)!}.
  \]
  Finally, rewriting the equation using $k = n - r$ we obtain
  \[
    \frac{n!}{n^n}\sum_{r = 1}^n \frac{r^{n - r}}{(n - r)!} = 
    \frac{n!}{n^n}\sum_{k = 0}^{n - 1} \frac{(n - k)^k}{k!}, 
  \]
  as required.
\end{proof}

In order to prove that $\mathbb{G}_n \to 0$ as $n \to \infty$ we use an
auxiliary function for which we prove some analytical properties. Also recall
that $\Omega \in \mathbb{R}$ is a unique constant which satisfies $\Omega
e^\Omega = 1$.

\begin{lem}\label{lem:auxilary-func}
  Let $F:(0,1) \to \R$ be given by $F(x) = x\log(x^{-1} -1)+x$. Then $F$  has a
  unique maximum at $\alpha = \frac{\Omega}{1 + \Omega} \in (0,1)$ and
  $F(\alpha) = \Omega < 1$. 
\end{lem}

\begin{proof}
  First observe that $F(x)$ is continuous on $(0,1)$, and $F(x) \to 0$ as $x \to 0$ and  $F(x) \to -\infty$ as $x
  \to 1$. The first and second derivative are continuous and given by
  \[
    \frac{dF(x)}{dx}=1-\frac{1}{1-x}+\log(x^{-1}-1)
    \quad \text{and} \quad
    \frac{d^2F(x)}{dx^2}=-\frac{1}{(x-1)^2 x}.
  \]
  Clearly, $\frac{d^2F(x)}{dx^2}<0$ for all $x \in (0,1)$, but
  $\frac{dF(x)}{dx} \to \infty$ as $x\to0$ and so the derivative is positive in
  a neighbourhood of $0$. But $F(x)\to-\infty$ as $x\to1$ and thus $F$ has a
  unique maximum at $\alpha$ implicitly given by
  \[
    1 - \frac{1}{1 - \alpha} + \log\left(\frac{1 - \alpha}{\alpha}\right) = 0,
  \]
  or in other words
  \[
    \frac{\alpha}{1 - \alpha} = \log\left(\frac{1 - \alpha}{\alpha}\right).
  \]
  It then follows that $\frac{\alpha}{1 - \alpha} = \Omega$, by the definition
  of $\Omega$. Hence $\alpha = \frac{\Omega}{1 + \Omega}$ and
  \[
    F(\alpha) = \alpha \log\left(\frac{1 - \alpha}{\alpha}\right) + \alpha = \alpha \left( 1 +
    \frac{\alpha}{1 - \alpha}\right) = \frac{\alpha}{1 - \alpha} = \Omega.
    \qedhere
  \]
\end{proof}

Finally, we conclude this section by describing the asymptotic behaviour of
$\mathbb{G}_n$.
\begin{prop}
  The probability $\mathbb{G}_n$, that $\langle x \rangle$ is a group where $x \in
  \mathcal{T}_n$ is chosen with uniform distribution, tends to $0$
  exponentially at the rate less than $1 - \Omega$.
\end{prop}
\begin{proof}
  By Lemma~\ref{lem:1-lower-bound}
  \[
    \mathbb{G}_n =  \frac{n!}{n^n} \sum_{k = 0}^{n - 1} \frac{(n - k)^k}{k!}.
  \]
  We use the Stirling approximation $\log n! =n\log n - n + O(\log(n))$. Then
  \[
    \frac{\log \mathbb{G}_n}{n} = n^{-1} O(\log(n)) - 1 + n^{-1} \log \sum_{k = 0}^{n-1}
    \frac{(n - k)^k}{k!}.
  \]
  Note that the last term can be bounded from above and below in the following way 
  \[
    \log \left( \max_{k \in \{0,\dots, n-1\}} \frac{(n - k)^k}{k!} \right) \leq
    \log \sum_{k=0}^{n-1} \frac{(n - k)^k}{k!} \leq \log \left( n \max_{k \in
    \{0,\dots, n-1\}} \frac{(n - k)^k}{k!} \right).
  \]
  Hence
  \[
    \log \sum_{k=0}^{n-1} \frac{(n - k)^k}{k!} = \log \left(
    \max_{k\in\{0,\dots, n-1\}} \frac{(n - k)^k}{k!} \right) + O(\log n),
  \]
  and so
  \[
    \lim_{n \to \infty} \frac{\log \mathbb{G}_n}{n} = -1 + \lim_{n \to \infty} n^{-1}
    \log \left( \max_{k\in\{0,\dots, n-1\}} \frac{(n - k)^k}{k!} \right).
  \]
  Considering the second term in the above equation, noting that for $n \geq 3$ 
  the maximum does not occur at $k = 0$, it follows that
  \begin{align*}
    n^{-1}\log \left( \max_{k \in \{0,\dots, n-1\}} \frac{(n - k)^k}{k!}
    \right)
    = & \max_{k \in \{1,\dots, n-1\}} n^{-1} \log \left( \frac{(n -
    k)^k}{k!} \right)\\
    = &\max_{k \in \{1,\dots, n-1\}} n^{-1} \big( k \log(n-k) - k \log k + k \\ 
       &- O(\log k) \big)\\
    = & \max_{x \in M_n} \left( x \log(x^{-1} - 1) + x \right) -
      n^{-1}O\left(\log n \right),
  \end{align*}
  where $M_n=\{ \frac{1}{n}, \frac{2}{n}, \dots, \frac{n-1}{n}\}$. Since $F$ is
  continuous on $(0,1)$ we conclude that $\max_{x \in M_n}
  F(x)\to\max_{x\in(0,1)}F(x)$ as $n\to\infty$. Therefore
  \[
    \lim_{n\to{\infty}}\frac{\log \mathbb{G}_n}{n} =
    -1 + \lim_{n \to \infty} n^{-1}\log \left( \max_{k \in \{0,\dots, n-1\}}
    \frac{(n - k)^k}{k!} \right) = F(\alpha) - 1 = \Omega - 1 < 0,
  \]
  by Lemma~\ref{lem:auxilary-func} as required.
\end{proof}

\subsection{$\mathbb{T}_n$ tends to zero}

Recall that for $n, r\in \mathbb{N}$ such that $r \leq n$, $\mathcal{A}(n,r)$
denotes the set of partitions of $\{1,\ldots,n\}$ into $r$ non-empty
components. Similarly, define $\mathcal{B}(n,r)$ to be the set of 
subsets of $\{1, \ldots, n\}$ of cardinality $r$. Then $\lvert\mathcal{B}(n,r)\rvert =
\binom{n}{r}$.

\begin{lem}\label{lem:prob-equal-ranks}
  Let $n \in \mathbb{N}$. Then the probability that $\rank(xyx) = \rank(y)$,
  where $x, y \in \mathcal{T}_n$ are chosen with uniform probability, is
  \[
    \mathbb{T}_n = \frac{1}{n^{2n}} \sum_{r = 1}^n  \binom{n}{r} r!
    \sum_{k = 1}^r \stirling{r}{k} k! k^{n - r} \sum_{\{A_1, \ldots, A_r\} \in
    \mathcal{A}(n,r)} \sum_{B \in \mathcal{B}(r, k)} \prod_{i \in B} |A_i|.
  \]
\end{lem}

\begin{proof}
  Let $x, y \in \mathcal{T}_n$ be such that $\rank(xyx) = \rank(y)$. We first
  show that $\im(xy)$ is contained in a transversal of $x$. Let $\mathfrak{T}$
  be a transversal of $xyx$. Then $xyx$ is injective on $\mathfrak{T}$ by
  definition, and so $x$ is injective on $(\mathfrak{T})xy$. Hence $\im(xy) =
  (\mathfrak{T})xy$ is contained in a transversal of $x$.

  Suppose that $\rank(x) = r$, $\rank(y) = k$, and $\{A_1, \ldots, A_r\} \in
  \mathcal{A}(n,r)$ are the kernel classes of $x$. Then there are $\binom{n}{r}
  r!$ choices for $x$. Since 
  \[
    \rank(y) \geq \rank(xy) \geq \rank(xyx) = \rank(y),
  \]
  it follows that $\rank(xy) = \rank(y) = k$, and also $\im(y) = \im(xy)$.
  Since $x$ is injective on $\im(xy)$, there are
  \[
    \sum_{B \in \mathcal{B}(r, k)} \prod_{i \in B} |A_i| 
  \]
  choices for $\im(y) = \im(xy)$. That is, $\im(xy)$ contains at most one
  point from any kernel class of $x$. Since $(\im(x))y = \im(xy) = \im(y)$, there are
  $\stirling{r}{k} k!$ ways for $y$ to map $\im(x)$ to $\im(y)$. Finally,
  $(\{1, \ldots, n\}) \setminus \im(x))y\subseteq \im(y)$, and so there $k^{n -
  r}$ for $y$ to map $(\{1, \ldots, n\}) \setminus \im(x))$ to $\im(y)$. Hence
  there are in total
  \[
    \stirling{r}{k} k! k^{n - r}\sum_{B \in \mathcal{B}(r, k)} \prod_{i \in B} |A_i| 
  \]
  choices for $y$. Therefore
  \[
    \mathbb{T}_n = \frac{1}{n^{2n}} \sum_{r = 1}^n \sum_{k = 1}^r \sum_{\{A_1,
    \ldots, A_r\} \in \mathcal{A}(n,r)} \binom{n}{r} r!  \stirling{r}{k} k!
    k^{n - r}\sum_{B \in \mathcal{B}(r, k)} \prod_{i \in B} |A_i|, 
  \]
  since $|\mathcal{T}_n| = n^n$.
\end{proof}

Next, we simplify the expression for $\mathbb{T}_n$.
\begin{lem}\label{lem:double-sum}
  Let $n, r, k \in \mathbb{N}$ such that $k \leq r \leq n$. Then
  \[
   \sum_{\{A_1, \ldots, A_r\} \in \mathcal{A}(n,r)} \sum_{B \in \mathcal{B}(r,
   k)} \prod_{i \in B} |A_i| = \sum_{s = k}^{n + k - r} \binom{n}{s}
   \stirling{n - s}{r - k} \binom{s}{k} k^{s-k}.
  \]
\end{lem}

\begin{proof}
  Let $B \in \mathcal{B}(r,k)$ and $\{A_1, \ldots, A_r\} \in \mathcal{A}(n,r)$
  be fixed and denote the number $|\bigcup \{A_b : b \in B\}|$ by $s$. Note
  that every $A_i$ is non-empty, so $k \leq s \leq n - (r - k)$. Now suppose
  that only $B$ is fixed, then for every value of $s \in \{k, \ldots, n + k -
  r\}$, there are $\binom{n}{s}$ choices for $\bigcup \{A_b : b \in B\}$, and
  there are $\stirling{n-s}{r-k}$ many choices to choose $\{A_b : b \notin
  B\}$. Hence we can write
  \[
    \sum_{\{A_1, \ldots, A_r\} \in \mathcal{A}(n,r)} \sum_{B \in \mathcal{B}(r,
    k)} \prod_{i \in B} |A_i| 
     =  
    \sum_{s = k}^{n + k - r} \binom{n}{s} \stirling{n - s}{r - k}
    \sum_{\{A_1, \ldots, A_k\} \in \mathcal{A}(s, k)} \prod_{i = 1}^k |A_i|.
  \] 

  The result follows by Lemma~\ref{lem:idempotent-numbers}.
\end{proof}

Finally, we prove the main lemma of this section.

\begin{lem}\label{lem:prob-approximation}
  There exist $r \in(0,1)$ and $c > 0$ such that $\mathbb{T}_n  \leq c
  n^{7/2}r^n$.
\end{lem}

\begin{proof}
  Note that by Stirling's approximation there are constants $a, b> 0$ such that
  $a n^n e^{-n} \leq n! \leq b  n^{n + \frac{1}{2}} e^{-n}$ for all $n\in\N$.

  It follows from
  Lemmas~\ref{lem:stirling-upper},~\ref{lem:prob-equal-ranks},
  and~\ref{lem:double-sum} that
  \begin{align*}
    n^{2n} \mathbb{T}_n 
    =&
    \sum_{r = 1}^n \binom{n}{r} r! \sum_{k=1}^{r} \stirling{r}{k} k!\,
    k^{n-r} \sum_{s=k}^{n+k-r} \binom{n}{s} \stirling{n-s}{r-k}
    \binom{s}{k}k^{s-k}
    \\ \leq&
    \sum_{r=1}^n \binom{n}{r} r! \sum_{k=1}^{r} \binom{r}{k} k^{n-k} k!
    \sum_{s=k}^{n+k-r} \binom{n}{s} \binom{s}{k} \binom{n-s}{r-k}
    (r-k)^{n-s-r+k} k^{s-k}.
  \end{align*}
  Observe that 
  \begin{equation}\label{equation:3}
    \binom{n}{s} \binom{s}{k} = \binom{n}{k} \binom{n - k}{n - s}
    \quad \text{and} \quad
    \binom{n - k}{n - s} \binom{n - s}{r - k} = \binom{n - k}{r - k} \binom{n - r}{s - k}.
  \end{equation}
  Hence
  \begin{align*}
    n^{2n} \mathbb{T}_n 
    \leq&
    \sum_{r=1}^n \binom{n}{r} r! \sum_{k=1}^{r} \binom{r}{k} k^{n-k} k!
    \sum_{s=k}^{n+k-r} \binom{n}{k} \binom{n - k}{r - k} \binom{n - r}{s - k}
    (r-k)^{n-s-r+k} k^{s-k}
    \\ =&
    \sum_{r=1}^n \binom{n}{r} r! \sum_{k=1}^{r} \binom{r}{k} k^{n-k} k!
    \binom{n}{k} \binom{n - k}{r - k} \sum_{i = 0}^{n-r} \binom{n - r}{i}
    (r-k)^{n-r - i} k^{i}
    \\ =&
    \sum_{r=1}^n \binom{n}{r} r! \sum_{k=1}^{r} \binom{r}{k} k^{n-k} k!
    \binom{n}{k} \binom{n - k}{r - k}
    r^{n - r}.
  \end{align*}

  It can also be show that
  \begin{equation}\label{equation:4}
    \binom{n}{k} \binom{n - k}{r - k} = \binom{n}{r} \binom{r}{k},
  \end{equation}
  and so
  \[
    n^{2n} \mathbb{T}_n \leq \sum_{r=1}^n \binom{n}{r}^2 r! \sum_{k=1}^{r}
    \binom{r}{k}^2 k! \, k^{n-k} r^{n - r} = \sum_{r=1}^n \frac{n!^2}{(n -
    r)!^2} \sum_{k=1}^{r} \frac{r!}{k!(r-k)!^2}  k^{n-k} r^{n - r}.
  \]
  Hence using Stirling's formula there is a constant $c > 0$ such that
  \[
    n^{2n} \mathbb{T}_n \leq c \sum_{r = 1}^n \frac{n^{2n + 1} e^{-2n}}{(n -
    r)^{2(n-r)} e^{-2(n-r)}} \sum_{k = 1}^r \frac{r^{r + \frac{1}{2}} e^{-r}}
    {e^{-k} k^{k} e^{-2(r - k)} (r - k)^{2(r - k)}} k^{n -k} r^{n - r},
  \]
 which can be simplified to
  \begin{align*}
    \mathbb{T}_n &\leq c \sum_{r = 1}^n \sum_{k = 1}^r \frac{n r^{n + \frac{1}{2}
    } k^{n - 2k}} {e^{r + k}(n - r)^{2(n-r)} (r - k)^{2(r - k)}} \\
    &\leq
    c n^2 \max_{\substack{1\leq r \leq n\\ 1\leq k \leq r}} \left\{ \frac{n r^{n +
    \frac{1}{2}} k^{n - 2k}} {e^{r + k}(n - r)^{2(n-r)} (r - k)^{2(r - k)}}
    \right\}.
  \end{align*}

  Let $x, y \in [0,1]$ be such that $r = xn$ and $k = yr = xyn$. Then
  \begin{align*}
    \mathbb{T}_n 
    &\leq c n^2 \max_{\substack{1\leq r \leq n\\1\leq k \leq r}} \left\{
    n^{\frac{3}{2} } \frac{x^{n + \frac{1}{2} } (xy)^{n - 2k}} {e^{r + k}(1 -
    x)^{2(n-r)} (x - xy)^{2(r - k)}} \right\} \\ 
    &\leq c n^{\frac{7}{2}} \sup_{(x,y) \in [0,1]^2} \left\{ \frac{x^{2(n - xn) +
    \frac{1}{2}} y^{n - 2xyn}} {e^{xn + xyn}(1 - x)^{2(n-xn)} (1 - y)^{2(xn -
    xyn)}} \right\} \\ 
    &\leq c n^{\frac{7}{2}} \left( \sup_{(x,y) \in [0,1]^2} \left\{ \frac{x^{2(1 - x)}
    y^{1 - 2xy}} {e^{x(1 + y)}(1 - x)^{2(1-x)} (1 - y)^{2x(1 - y)}} \right\}
    \right)^n.
  \end{align*}

  It only remains to show that the supremum in the above equation is less than
  $1$. In order to do so, define $F : [0,1]^2 \to \mathbb{R}$ by 
  \[
    F(x,y) = \frac{x^{2(1 - x)} y^{1 - 2xy}} {e^{x(1 + y)}(1 - x)^{2(1-x)} (1 -
    y)^{2x(1 - y)}}.
  \]
  Note that $F$ is continuous on a compact set $[0, 1]^2$, and so has a
  maximum. Hence we only need to consider the boundary of the domain and
  stationary points of $F$, that is points in $[0,1]^2$ where $\partial
  F/\partial x = 0 = \partial F/\partial y$. However, while it can be
  immediately be deduced from plots, using any mathematical software, that the
  maximum of $F$ is strictly less than $1$, we show it here analytically. To
  this end, define the functions $F_1, F_3: [0, 1] \to \mathbb{R}$ and $F_2:
  [0,1]^2 \to \mathbb{R}$ by
  \[
    F_1 (x) = \frac{x^{2(1 - x)}} {(1 - x)^{2(1-x)} },
    \qquad
    F_2(x,y)=\frac{y^{1 - 2xy}} {e^{x(1 + y)} (1 - y)^{2x(1 - y)}},
  \]
  and
  \[
    F_3(y) = -1 - y - 2 (1 - y) \log(1 - y) - 2 y \log y.
  \] 
  Then $F(x,y)=F_1(x) F_2(x,y)$, and it can be shown that $\partial
  F_2(x,y)/ \partial x = F_2(x,y) F_3(y)$. Also note that that $F_1$, $F_2$,
  and $F_3$ are all continuous.

  Since $F_1(x)$ is continuous on a compact set, we can perform standard
  analysis of stationary points. Then
  \[
    \frac{d F_1(x)}{dx} = F_1(x)  (1 + x \log(1 - x) - x \log x).
  \]
  and $F_1(x) > 0$ for all $x \in (0, 1]$. Thus the stationary points of
  $F_1$ are either $0$, $1$, or $x_0$, which is given by the equation 
  \[
    (1 + x_0 \log(1 - x_0) - x_0 \log x_0)=0,
  \]
  or in other words, $x_0=1/(1+W(e^{-1}))$ where $W$ is the Lambert-W
  function. It follows that $F_1$ is bounded from above by $\max\{F_1(0),
  F_1(1), F(x_0)\}$. A simple algebraic manipulation gives 
  \[
    F_1(x_{0})=W(e^{-1})^{-\frac{2}{1+W(e^{-1})^{-1}}}< 1.75.
  \]
  Since $F_1(0) = 0$ and $F_1(1) = 1$, it follows that $F_1(x) \leq 1.75$ for
  all $x \in [0,1]$. We also note here, that $dF_1(x)/dx$ is positive for all
  $x \in [0, x_0)$.

  Next we show that $\partial F_2(x,y)/\partial x \leq 0$ for all $x y \in
  [0, 1]$. First, observe that  $F_2 (x,y) \geq 0$ over $[0,1]^2$. Since
  $\partial F_2(x, y)/\partial x = F_2(x,y) F_3(y)$, we are left to show that
  $F_3(y) \leq 0$ for $y \in [0, 1]$. Note that $F_3(0) = -1$, $F_3(1) = -2$,
  and 
  \[ 
    \frac{d F_3(y)}{dy}=-1 + 2 \log(1 - y) - 2
    \log y \quad \text{ and }\quad
    \frac{d^2 F_3(y)}{dy^2}=\frac{2}{(y-1)y}.
  \]
  Since $d^2 F_3(y)/ dy^2<0$ for all $y\in(0,1)$, $F_3$ has a unique maximum at
  $(1+\sqrt{e})^{-1}$, and
  \[
    F_3\left(\frac{1}{1 + \sqrt{e}}\right) =2 \log(1 + \sqrt{e}) - 2<0.
  \]
  Hence $\partial F_2(x,y)/\partial x \leq 0$ for all $x, y \in [0, 1]$, and so
  $F_2(x,y)\leq F_2(0,y)=y$ and in particular $F_2(x,y)\leq1$.

  For the last step of the proof consider 
  \[
    F_2(\frac{1}{2},y)=e^{-\frac{y+1}{2}}\left(\frac{y}{1-y}\right)^{1-y}.
  \]
  Then 
  \[
    \frac{dF_2(\frac{1}{2},y)}{dy}= F_2(\frac{1}{2},y) (y^{-1} -
    \frac{1}{2} + \log\left(y^{-1} - 1\right) )
  \]
  and the derivative has a single root at $y_0=1/(1+W(e^{-1/2}))$. Hence if $x
  \in [1/2, 1]$ and $y \in [0,1]$, then 
  \[
    F_2(x,y) \leq F_2(1/2, y) \leq \max\{F_2(1/2,0),F_2(1/2,1),F_2(1/2,y_0)\} < 0.56, 
  \] 
  and so $F(x, y) \leq 1.75 \cdot 0.56 = 0.98$. Since $F(x, y)$ continuous on
  $[0,1]^2$ there is $\varepsilon > 0$ and $\beta < 1$ such that
  $F(x, y) \leq \beta$ for all $x \in [1/2 - \varepsilon, 1]$ and all $y \in
  [0, 1]$.

  Finally, recall that $x_0 = 1/(1 + W(e^{-1})) > 0.78$ and $F_1(x)$ is
  increasing on $[0, x_0)$.  We observe that if $x \in [0, 1/2 - \varepsilon]
  \subseteq [0, x_0)$ then $0 = F_1(0) \leq F_1(x) \leq  F_1(1/2 - \varepsilon)
  < F_1(1/2) = 1$. Since $F_2(x, y) \leq 1$, it follows that $F(x, y) \leq
  F_1(1/2 - \varepsilon) < 1$ for all $x \in [0, 1/2 - \varepsilon]$ and all $y
  \in [0, 1]$. Therefore $F(x, y) \leq \max\{\beta, F_1(1/2 - \varepsilon)\} <
  1$ for all $x, y \in [0, 1]$, as required.
\end{proof}

The following is an immediate corollary of Lemmas~\ref{lem:prob-equal-ranks}
and~\ref{lem:prob-approximation}.

\begin{cor}\label{cor:prob-k2}
  The probability $\mathbb{T}_n$, that $\rank(xyx) = \rank(y)$ where $x, y \in
  \mathcal{T}_n$ are chosen with uniform distribution, tends to $0$ as $n \to
  \infty$ exponentially fast.
\end{cor}

\subsection{$\mathbb{V}_n$ tends to zero}

We start by finding an expression for $\mathbb{V}_n$ in terms of $n$. The
argument is similar to the proof of Lemma~\ref{lem:prob-equal-ranks}.

\begin{lem}
  Let $n \in \mathbb{N}$. Then the probability that $\rank(xyz) = \rank(y)$,
  where $x, y, z \in \mathcal{T}_n$ are chosen with uniform probability, is
  \[
    \mathbb{V}_n = \frac{1}{n^{3n}} \sum_{r = 1}^n \sum_{k = 1}^r \sum_{t =
    1}^{\min(r, k)} \stirling{n}{r} \binom{n}{r} r! \binom{n}{k} k!
    \stirling{r}{t} t! t^{n - r} \sum_{s = t}^{n + t - k} \binom{n}{s} \binom{n
    - s}{r - k} \binom{s}{t} t^{s - t}.
  \] 
\end{lem}
\begin{proof}
  If $x, y, z \in \mathcal{T}_n$ are such that $\rank(xyz) = \rank(y)$. We first
  show that $\im(xy)$ is contained in a transversal of $z$. Let $\mathfrak{T}$
  be a transversal of $xyz$. Then $xyz$ is injective on $\mathfrak{T}$ by
  definition, and so $z$ is injective on $(\mathfrak{T})xy$. Hence $\im(xy) =
  (\mathfrak{T})xy$ is contained in a transversal of $z$.

  Suppose that $\rank(x) = r$, $\rank(z) = k$, $\rank(y) = t$, and $\{A_1,
  \ldots, A_k\} \in \mathcal{A}(n,k)$ are the kernel classes of $z$. Note that
  $t \leq r$ and $t \leq k$. Then there are $\binom{n}{r} \stirling{n}{r} r!$
  choices for $x$ and $\binom{n}{k} k!$ choices for $z$. Since 
  \[
    \rank(y) \geq \rank(xy) \geq \rank(xyz) = \rank(y),
  \]
  it follows that $\rank(xy) = \rank(y) = t$, and also $\im(y) = \im(xy)$.
  Since $z$ is injective on $\im(xy)$, there are
  \[
    \sum_{B \in \mathcal{B}(r, t)} \prod_{i \in B} |A_i| 
  \]
  choices for $\im(y) = \im(xy)$. That is, $\im(xy)$ contains at most one
  point from any kernel class of $z$. Since $(\im(x))y = \im(xy) = \im(y)$, there are
  $\stirling{r}{t} t!$ ways for $y$ to map $\im(x)$ to $\im(y)$. Finally,
  $(\{1, \ldots, n\}) \setminus \im(x))y\subseteq \im(y)$, and so there $t^{n -
  r}$ ways for $y$ to map $(\{1, \ldots, n\}) \setminus \im(x))$ to $\im(y)$. Hence
  there are in total
  \[
    \stirling{r}{t} t! t^{n - r}\sum_{B \in \mathcal{B}(r, t)} \prod_{i \in B} |A_i| 
  \]
  choices for $y$. Therefore
  \[
    \mathbb{V}_n = \frac{1}{n^{3n}} \sum_{r = 1}^n \sum_{k = 1}^r \sum_{t =
    1}^{\min(r, k)} \sum_{\{A_1, \ldots, A_k\} \in \mathcal{A}(n,k)}
    \binom{n}{r} r! \binom{n}{k} k! \stirling{n}{r} \stirling{r}{t} t! t^{n -
    r}\sum_{B \in \mathcal{B}(r, t)} \prod_{i \in B} |A_i|, 
  \]
  since $|\mathcal{T}_n| = n^n$. It follows from Lemma~\ref{lem:double-sum}
  that
  \[
    \mathbb{V}_n = \frac{1}{n^{3n}} \sum_{r = 1}^n \sum_{k = 1}^r \sum_{t =
    1}^{\min(r, k)} \stirling{n}{r} \binom{n}{r} r! \binom{n}{k} k!
    \stirling{r}{t} t! t^{n - r} \sum_{s = t}^{n + t - k} \binom{n}{s} \binom{n
    - s}{r - k} \binom{s}{t} t^{s - t},
  \]
  as required.
\end{proof}

Finally, we prove the main two lemmas of this section. This is an analogue of
Lemma~\ref{lem:prob-approximation}.

\begin{lem}\label{lem:bound-V}
  There exist $c > 0$ such that 
  \[
    \mathbb{V}_n \leq c n^5\left(\max_{\substack{x,y,z\in
    (0,1]\\z\leq\min(x,y)}}G(x,y,z)\right)^n,
  \]
  where
  \[
    G(x,y,z) = \frac{x^{1-x} y^{1-y} z^{1-2z}}{e^{x+y+z} (1-x)^{2(1-x)}
    (1-y)^{2(1-y)} (x-z)^{x-z} (y-z)^{y-z}}.
  \]
\end{lem}
\begin{proof}
  We begin by applying the same strategy as in
  Lemma~\ref{lem:prob-approximation}. That is we use Lemma~\ref{lem:stirling-upper} to
  give an upper bound without Stirling numbers of the second kind and then use
  equations \eqref{equation:3} and \eqref{equation:4}. It follows that

  \begin{align*}
    n^{3n}\mathbb{V}_n &=
      \sum_{r = 1}^n \sum_{k = 1}^r \sum_{t = 1}^{\min(r, k)}
      \stirling{n}{r} \binom{n}{r} r! \binom{n}{k} k! \stirling{r}{t} t! t^{n -
      r} \sum_{s = t}^{n + t - k} \binom{n}{s} \binom{n - s}{r - k}
      \binom{s}{t} t^{s - t} \\
    &\leq
      \sum_{r = 1}^n \sum_{k=1}^n  \sum_{t=1}^{\min(r,k)} \left(\frac{n!}{(n-r)!}\right)^2
       \left( \frac{n!}{(n-k)!}\right)^2
      \frac{t^{n-t} r^{n-r} k^{n-k}}{(r-t)!(k-t)!t!}
  \end{align*}
  Replacing the sums with $n$ times their maximal value we obtain after
  some algebraic manipulation
  \[
    n^{3n}\mathbb{V}_n \leq 
    n^3 \max_{\substack{1\leq r\leq n\\ 1\leq k\leq n\\
    1 \leq t \leq \min(r,k)}}
    \left\{ \frac{(n!)^4 k^{n-k} r^{n-r} t^{n -
    t}}{((n-r)!(n-k)!)^2(r-t)!(k-t)!t!}\right\}
  \] 

  Using Stirling's approximation $\mathbb{V}_n$ can be bounded by
  \[
    n^{3n}\mathbb{V}_n \leq c n^3 \max_{\substack{1\leq r\leq n\\ 1\leq k\leq
    n\\ 1 \leq t \leq \min(r,k)}} \left\{ \frac{n^{4n + 2} k^{n-k} r^{n-r} t^{n
    - 2t}}{(n - r)^{2(n - r)} (n - k)^{2(n - k)} (r - t)^{r - t} (k - t)^{k - t}
    e^{r + k + t}} \right\}
  \]
  for some $c > 0$. Let $x=r/n$, $y=k/n$, and $z=t/n$. The above equation can
  be rearranged to obtain
  \[
    \mathbb{V}_n \leq
      c n^5\left(\max_{\substack{x,y,z\in (0,1]\\z\leq\min(x,y)}}
      \left\{ \frac{x^{1-x}
      y^{1-y}z^{1-2z}}{e^{x + y + z}(1-x)^{2(1-x)}(1-y)^{2(1-y)} 
      (x-z)^{x-z}(y-z)^{y-z}}\right\} \right)^n.
  \]
  Hence
  \[
    \mathbb{V}_n \leq c n^5\left(\max_{\substack{x,y,z\in
    (0,1]\\z\leq\min(x,y)}}G(x,y,z)\right)^n,
  \]
  as required.
\end{proof}

By inspection we see that $G$ is continuous and bounded on 
\[
  X=\{(x, y, z)\in \R^3 \mid 0<x, y<1 \, \text{and} \, 0 < z < \min(x,y)\}.
\]
We can further extend the definition of $G$ to the closure $\overline X$. It
remains to find the maximum of $G$, which we do in the last lemma of this
section.

\begin{lem}\label{lem:bound-G}
  There exists $r \in (0, 1)$ such that $G(x, y, z) \leq r$ for all $x, y, z
  \in [0, 1]$ such that $z \leq \min(x, y)$.
\end{lem}

\begin{proof}
  First we establish the value of $G$ on the boundary $\overline{X}\setminus
  X$. Clearly for either $x=0$, $y = 0$ or $z = 0$ we have $G(x,y,z)=0$. If $x=1$
  \[
    G(1,y,z)=\frac{ y^{1-y}z^{1-2z}}{e^{1+y+z}(1-y)^{2(1-y)}(1-z)^{1-z}(y-z)^{y-z}}.
  \]
  By considering the derivative of $x^{-x}$, we can show that $x^{-x} \leq
  e^{e^{-1}}$ for all $x \in [0, 1]$. Hence 
  \[
    G(1,y,z)= e^{7 e^{-1} - 1} \cdot \frac{ y z}{e^{y + z}}.
  \]
  Also note that $x \to x e^{-x}$ is increasing on $[0, 1]$, and so $x e^{-x} \leq e^{-1}$.
  Therefore
  \[
    G(1, y, z) \leq e^{7e^{-1} - 3} \leq 0.7.
  \]
  By symmetry this also holds for $y=1$.
  
  Let $x=z$. Then
  \[
    G(x,y,x)=\frac{x^{2-3x} y^{1-y}}{e^{2x+y}(1-x)^{2(1-x)}(1-y)^{2(1-y)}(y-x)^{y-x}}.
  \]
  Similarly, 
  \[
    G(x,y,x) \leq e^{3e^{-1}} \left(\frac{x^{1 - x}}{e^x (1 - x)^{1 -x }}
    \right)^2 \cdot \frac{y^{1-y}}{e^{y}(1-y)^{1-y}}.
  \]
  By considering the derivatives of $x \to \frac{x^{1 - x}}{e^x (1 - x)^{1 -x
  }}$, we can show that the function has a unique maximum at $x_0 = \frac{1}{1 +
  \Omega}$. Hence after some algebraic manipulations we get
  \[
    \frac{x_0^{1 - x_0}}{e^{x_0} (1 - x_0)^{1 -x_0 }} = e^{- \frac{1}{1 +
    \Omega}} \cdot \Omega^{- \frac{\Omega}{1 + \Omega}}= e^{\Omega - 1},
  \]
  and so
  \[
    G(x, y, x) \leq e^{3(e^{-1} + \Omega - 1)} < 1.
  \]
  By symmetry the same holds for $y = z$. 

  The partial derivatives of $G$ are as follows
  \begin{align*}
    \frac{\partial G(x,y,z)}{\partial x} &= 
      \left(\log\left(\frac{(1 - x)^2}{x^2 - xz}\right) + \frac{1-x}{x}\right)
      G(x,y,z),\\
    \frac{\partial G(x,y,z)}{\partial y} &= 
      \left(\log\left(\frac{(1 - y)^2}{y^2 - yz}\right) + \frac{1-y}{y}\right)
      G(x,y,z),\\
    \frac{\partial G(x,y,z)}{\partial z} &= 
      \left(\log\left(\frac{(x - z)(y - z)}{z^2}\right) + \frac{1 - z}{z} 
      \right) G(x,y,z).
  \end{align*} 
  Suppose that $(\alpha, \beta, \gamma) \in (0,1]^3$ is a stationary point of
  $G(x, y, z)$. Note that $G(x, y, z) > 0$ if $x, y, z \neq 0$, and so
  \[
      \log\left(\frac{(1 - \alpha)^2}{\alpha^2 - \alpha \gamma}\right) +
      \frac{1-\alpha}{\alpha} = 0 
      \quad \text{and} \quad
      \log\left(\frac{(1 - \beta)^2}{\beta^2 - \beta \gamma}\right) +
      \frac{1-\beta}{\beta} = 0
  \]
  Hence
  \begin{equation}\label{equation:5}
    \gamma = \alpha - \frac{ e^{\alpha^{-1}}(1 - \alpha)^2}{e \alpha}
     = \beta - \frac{ e^{\beta^{-1}}(1 - \beta)^2}{e \beta}
  \end{equation}
  However, the function $x \to x - \frac{ e^{x^{-1}}(x - 1)^2}{e x}$ is
  increasing and thus injective, implying that $\alpha = \beta$. Hence all
  stationary points of $G(x, y, z)$ are of the form $(\alpha, \alpha, \gamma)$.

  Substituting $\alpha = \beta$ into $\partial G(x,y,z)/\partial z=0$ and
  rearranging we obtain that 
  \[
    \frac{1}{\gamma} + 2\log(\frac{\alpha - \gamma}{\gamma}) = 1. 
  \]
  Combining the above equation with \eqref{equation:5} we get that $\alpha$
  satisfies 
  \begin{equation}\label{equation:7}
    \frac{e \alpha}{e \alpha^2 -e^{\alpha^{-1}}(1 - \alpha)^2} + 2
    \log\left(\frac{e^{\alpha^{-1}}(1 - \alpha )^2}{e
    \alpha^2-e^{\alpha^{-1}}(1 - \alpha)^2}\right)=1.
  \end{equation}
  It can be shown that the derivative of the function given by the left
  hand side of the above equation is
  \[
    D(x) = - \frac{e \left( e x^3(x + 3) - e^{x^{-1}} (1 - x)^2 ( x^2 + 2x - 1)
    \right) }{(1-x) \left( ex^2 - e^{x^{-1}} ( 1 -x)^2 \right)^2}.
  \]

  Note that since the function in \eqref{equation:5} is increasing, and so if
  $\alpha \leq 0.587$, then
  \[
    \gamma \leq 0.587 - \frac{ e^{0.587^{-1}}(1 - 0.587)^2}{e \cdot 0.587} < 0,
  \] 
  which contradicts $\gamma \in (0, 1]$. Hence $\alpha > 0.587$. It is easy to
  see that $x \to e^{x^{-1}} (1 - x)^2$ is decreasing for $x \in (0, 1)$
  and that $x \to x^2 + 2x - 1$ is increasing for $x \geq -1$. Thus 
  \[
    e^{x^{-1}} (1 - x)^2 ( x^2 + 2x - 1) \leq 2 e^{0.587^{-1}} ( 1 - 0.587)^2   <
    1.88 
  \]
  for $x \in [0.587, 1]$. On the other hand, $x \to e x^3(x + 3)$ is
  increasing, and so for $x \geq 0.587$
  \[
    e x^3 (x + 3) \geq e \cdot 0.587^3 (0.587 + 3) > 1.97.
  \]
  Therefore
  \[
    e x^3(x + 3) - e^{x^{-1}} (1 - x)^2 ( x^2 + 2x - 1) > 0 
  \]
  for $x > 0.587$, and thus $D(x) < 0$, implying that the left handside of
  \eqref{equation:7} is strictly decreasing. Hence there is a unique value
  $\alpha$ satisfying the \eqref{equation:7}. Moreover, we can see by
  inspection that $0.68152 < \alpha < 0.68153$. Since \eqref{equation:5} is
  strictly increasing, it also follows that $0.44403 < \gamma < 0.44407$. Finally
  \begin{align*}
    G(\alpha, \alpha, \gamma) &=
      \frac{\alpha^{2(1-\alpha)} \gamma^{1-2\gamma}}{e^{2\alpha+\gamma}
      (1-\alpha)^{4(1-\alpha)} (\alpha-\gamma)^{2(\alpha-\gamma})} \\ 
    &\leq
      \frac{0.68153^{2 (1-0.68153)} 0.44407^{1 - 2 \cdot 0.44407}}{(1-0.68153)^{4
      (1 - 0.68152)} (0.68152 -0.44407)^{2 (0.68153\, -0.44403)} e^{2 \cdot
      0.68152+0.44403}} \\
    &< 
      0.999.
  \end{align*}
  Therefor there exists $r \in (0, 1)$ such that $G(x, y, z) \leq r$ for all $x, y, z
  \in [0, 1]$ such that $z \leq \min(x, y)$.
\end{proof}

The following is an immediate corollary of Lemmas~\ref{lem:bound-V}
and~\ref{lem:bound-G}, which concludes the proof of Theorems~\ref{thm:main}
and~\ref{thm:smallgen}.

\begin{cor}
  The probability $\mathbb{V}_n$, that $\rank(xyz) = \rank(y)$ where $x, y, z \in
  \mathcal{T}_n$ are chosen with uniform distribution, tends to $0$ as $n \to
  \infty$ exponentially fast.
\end{cor}

\end{document}